\documentclass[12pt,reqno]{amsart}
\usepackage{etex}
\usepackage{amssymb,tikz-cd}
\usepackage{mathrsfs}
\usepackage{palatino}
\usepackage{bbm}
\usepackage{pdflscape}
\usepackage{mathrsfs}
\usepackage{cases}
\usepackage{epic}
\usepackage{amsfonts}
\usepackage{graphicx}
\usepackage{amsmath}
\usepackage{amssymb, upgreek}
\usepackage{bm}
\usepackage{latexsym,todonotes}
\usepackage{pdflscape}
\usepackage[all]{xypic}
\usepackage{color}
\usepackage{colordvi}
\usepackage{multicol}
\usepackage[normalem]{ulem}
\usepackage[linktocpage=true]{hyperref}
\usepackage{hyperref}
\hypersetup{colorlinks,linkcolor=blue,urlcolor=cyan,citecolor=red}
\usepackage{pictex}
\usepackage{amscd}
\usepackage{latexsym}
\usepackage{amssymb}
\usepackage{eucal}
\usepackage{eufrak}
\usepackage{verbatim}

\newcommand{\ZZ}{\mathbb{Z}}

\newcommand{\CC}{\mathcal{C}}

\newcommand{\cZ}{\mathcal{Z}}

\newcommand{\fb}{\mathfrak{b}}

\newcommand{\fh}{\mathfrak{h}}

\newcommand{\fn}{\mathfrak{n}}

\newcommand{\fq}{\mathfrak{q}}

\newcommand{\gl}{\mathfrak{gl}}

\newcommand{\Yq}{\Y(\mathfrak{q}_1)}

\DeclareMathOperator{\ad}{ad}

\DeclareMathOperator{\ev}{ev}

\DeclareMathOperator{\End}{End}
\DeclareMathOperator{\odd}{odd}

\DeclareMathOperator{\Id}{Id}

\DeclareMathOperator{\Mat}{Mat}

\DeclareMathOperator{\Y}{Y}

\DeclareMathOperator{\U}{U}
\numberwithin{equation}{section}
\newtheorem{Theorem}{Theorem}[section]
\newtheorem{Lemma}[Theorem]{Lemma}
\newtheorem{Corollary}[Theorem]{Corollary}
\newtheorem{Proposition}[Theorem]{Proposition}
\newtheorem{Remark}[Theorem]{Remark}

\theoremstyle{Theorem}

\newtheorem*{thm*}{Theorem}
\newtheorem*{thm**}{Corollary}
\newtheorem*{thm***}{Theorem B}
\theoremstyle{remark}

\numberwithin{equation}{section}

\begin{document}
\title[]{A note on the center of the queer super Yangian $\Y(\fq_1)$}
\author[Hao Chang \lowercase{and} Hongmei Hu]{Hao Chang \lowercase{and} Hongmei Hu*}
\address[Hao Chang]{School of Mathematics and Statistics, Central China Normal University, Wuhan 430079, China}
\email{chang@ccnu.edu.cn}
\address[Hongmei Hu]{Department of Mathematics, Shanghai Maritime University, Shanghai 201306, China}
\email{hmhu@shmtu.edu.cn}
\date{\today}
\thanks{*~Corresponding author.}

\begin{abstract}
We investigate two families of central elements of the queer super Yangian $\Yq$,
arising from the constructions of Poletaeva-Serganova and from Nazarov's quantum Berezinian.
We establish an explicit relation between their generating
series, thereby giving an answer to the question raised by Nazarov
[M. Nazarov, J. Math. Phys. \textbf{63}, 081702 (2022)].
\end{abstract}
\maketitle
\section{Introduction}
The Yangian $\Y(\fq_n)$ of the queer Lie superalgebra $\fq_n$ was introduced by Nazarov \cite{Na92, Na99}.
It can be viewed as a Hopf algebra deformation of the enveloping algebra of the twisted current Lie superalgebra associated with $\fq_n$.
Nazarov also constructed a central series $\cZ(u)$ for $\Y(\fq_n)$ and showed that its nonconstant even coefficients freely generate the center of the Yangian (see \cite[Section 3]{Na99}).

Finite $W$-algebras associated with queer Lie superalgebras are closely related to the Yangian $\Yq$.
In \cite{PS16},
Poletaeva and Serganova established a connection between the (principal) finite $W$-algebras for $\fq_n$ and the Yangian $\Yq$. Their work was further extended to the non-regular case in \cite{PS17}, where the finite $W$-algebras associated with non-regular nilpotent elements were studied in relation to the Yangian.
In their subsequent study of representations of the principal finite $W$-algebra and the Yangian $\Yq$,
Poletaeva and Serganova introduced a distinguished family of central elements of $\Yq$, see \cite[Lemma 5.1]{PS21}.
The Harish--Chandra images of these elements admit an explicit description in terms of symmetric polynomials.

Another family of central elements arises from Nazarov's quantum Berezinian.
In \cite{Na22}, Nazarov studied the queer Yangian $\Yq$ and introduced the quantum Berezinian $C(u)$, whose odd coefficients freely generate the center of $\Yq$.
More recently,
a new expression for the central series of $\Y(\fq_n)$ in terms of its Gauss generators was obtained in \cite{CW25}.

The purpose of this note is to determine explicitly the relation between the central elements constructed by Poletaeva and Serganova and those arising from Nazarov's quantum Berezinian.
Our approach combines the quotient homomorphisms from $\Yq$ to the finite $W$-algebras with Nazarov's tensor-evaluation homomorphisms. After passing to the Harish--Chandra realization,
the Poletaeva--Serganova central elements are described by symmetric polynomials,
while the image of the quantum Berezinian is described by a compatible matrix factorization.
By comparing these two descriptions, we obtain an explicit relation between the corresponding generating series.

The article is organized as follows. In Section \ref{pre},
we recall the necessary background on the queer Lie superalgebras, finite $W$-algebras and the super Yangian $\Yq$, together with the relevant quotient and evaluation homomorphisms.
In Section \ref{sec-center}, we compare the two families of central elements, namely, those constructed by Poletaeva and Serganova and those arising from Nazarov's quantum Berezinian, and derive an explicit generating series relation between them.
We conclude the section with a remark (Remark \ref{Remark}) on their connection with Nazarov's earlier central series $\cZ(u)$.

\section{Preliminaries}\label{pre}
\subsection{The Queer Lie superalgebra $\fq_n$}\label{Section:lie super qn}

Let the indices $i,j$ run through $\{\pm1,\ldots,\pm n\}$,
and put $\bar{i}=0$ if $i>0$ and $\bar{i}=1$ if $i<0$.
Equip the superspace $\mathbb{C}^{n|n}$ with the ordered
homogeneous basis
\[
  (v_1,\ldots,v_n,v_{-1},\ldots,v_{-n}),
\]
where $v_i$ is even for $i>0$ and odd for $i<0$.
We identify $\End\mathbb{C}^{n|n}$ with
the algebra of $2n\times 2n$ complex matrices relative
to this ordered basis. Thus the rows and columns are
indexed in the order $1,\ldots,n,-1,\ldots,-n$.
Let $E_{i,j}$ denote the matrix unit
defined by
\[
  E_{i,j}.v_k=\delta_{j,k}v_i.
\]
The algebra $\End\mathbb{C}^{n|n}$ is $\ZZ_2$-graded so that $\deg E_{i,j}=\bar i+\bar j$.
We will also regard $E_{i,j}$ as generators of the complex Lie superalgebra $\gl_{n|n}$.
The {\it queer} Lie superalgebra $\fq_n$ is the fixed point subalgebra in $\gl_{n|n}$ with respect to the involutive automorphism
\[
    E_{i,j}\mapsto E_{-i,-j}.
\]
It is easy to see that $\fq_n$ consists of matrices of the form:
\begin{align}\label{matrix form qn}
\begin{bmatrix}
A & B\\
B & A
\end{bmatrix},
\end{align}
where $A$ and $B$ are arbitrary $n\times n$ matrices.
The subalgebra $\fh_n$ consisting of matrices with $A,B$ being diagonal will be called the {standard Cartan sublagbera}.
By $\fn^+$ (respectively, $\fn^-$) we denote the
nilpotent subalgebras consisting of matrices with strictly upper triangular (respectively,
lower triangular) $A$ and $B$.
The Lie superalgebra $\fq_n$ has the triangular decomposition $\fq_n=\fn^-\oplus\fh_n\oplus\fn^+$ (see for instance \cite[1.2.6]{CW12}), and we set $\fb:=\fh_n\oplus \fn^+$.
\subsection{Finite $W$-algebra for $\fq_n$}
Now let us recall that the finite $W$-algebra $W^n$ associated with a principal even nilpotent element $\chi$ in the coadjoint representation of $\fq_n$ (see \cite{PS16, PS17}).
Note that a linear basis for $\fq_n$ consists of the following elements:
\begin{align}\label{basis for qn}
e_{i,j}:=E_{i,j}+E_{-i,-j},~f_{i,j}:=E_{i,-j}+E_{i-,j},\quad 1\leq i,j\leq n.
\end{align}
Choose $\chi\in\fq_n^*$ such that
\[
    \chi(f_{i,j})=0,~\chi(e_{i,j})=\delta_{i,j+1}.
\]
Let $I_{\chi}$ be the ideal in $\U(\fq_n)$ generated by $x-\chi(x)$ for all $x\in\fn^-$.
We denote by $\pi:\U(\fq_n)\rightarrow \U(\fq_n)/I_\chi$ the natural projection.
Then
\[
    W^n:=\{\pi(y)\in \U(\fq_n)/I_\chi;~(x-\chi(x))y\in I_\chi~{\rm for~all}~x\in\fn^-\}.
\]
Using identification of $\U(\fq_n)/I_\chi$ with the Whittaker module,
we can consider $W^n$ as a subalgebra of $\U(\fb)$ (cf. \cite[Section 2.1]{PS16}).
Let $\vartheta: \U(\fb)\rightarrow \U(\fh_n)$ be the canonical projection.
Its restriction to $W^n$ is the {\it Harish-Chandra homomorphism}
\begin{align}\label{HC-map}
\vartheta:W^n\rightarrow \U(\fh_n),
\end{align}
which is injective by \cite[Theorem 3.1]{PS16}.
Set
\[
\xi_i:=(-1)^{i+1}f_{i,i}, x_i:=\xi_i^2=e_{i,i}
\]
then
\begin{align}\label{uhn}
\U(\fh_n)\cong \mathbb{C}[\xi_1,\dots,\xi_n]/(\xi_i\xi_j+\xi_j\xi_i)_{1\leq i<j\leq n}.
\end{align}
\subsection{The super Yangian $\Yq$}\label{section:yangian}
The super Yangain $\Y(\fq_n)$ associated with the Lie superalgebra $\fq_n$ was introduced by M. Nazarov in \cite{Na92}.
In this paper, we only need the special case of $n=1$.
Recall that $\Yq$ is a complex associative unital algebra with a set of generators
\[
 t_{i,j}^{(r)},~{\rm where}~ r=1,2,\dots~{\rm and}~i,j=\pm 1.
\]
The $\ZZ_2$-grading of the algebra $\Yq$ is defined via $\deg t_{i,j}^{(r)}=\bar i+\bar j$.
To write down defining relations for these generators of $\Yq$,
we will use the formal power series in $u^{-1}$ with coefficients from $\Yq$,
\[
    t_{i,j}(u)=\delta_{i,j}+t_{i,j}^{(1)}u^{-1}+t_{i,j}^{(2)}u^{-2}+\cdots.
\]
Then for all possible indices $i,j,k,l$ we have the relations
\begin{align}\label{Tijtkl relation}
&[t_{i,j}(u),t_{k,l}(v)](-1)^{\bar i\bar k+\bar i\bar l+\bar k\bar l}\\ \nonumber
=&\frac{t_{k,j}(u)t_{i,l}(v)-t_{k,j}(v)t_{i,l}(u)}{u-v}-\frac{t_{-k,j}(u)t_{-i,l}(v)-t_{k,-j}(v)t_{i,-l}(u)}{u+v}(-1)^{\bar k+\bar l}
\end{align}
in $\Yq[[u^{-1},v^{-1}]]$.
The square brackets above stand for the supercommutator.
For all indices $i, j$ we also have the relations
\begin{align}\label{tij-u=t-i-iu}
t_{i,j}(-u)=t_{-i,-j}(u).
\end{align}
Then these power series can be collected together into a single matrix
\[
    T(u):=\sum\limits_{i,j}E_{i,j}\otimes t_{i,j}(u)\in \End\mathbb{C}^{1|1}\otimes \Yq[[u^{-1}]].
\]
Moreover, the element $T(u)$ of the algebra $\End\mathbb{C}^{1|1}\otimes \Yq[[u^-1]]$ is invertible, we put
\[
    T(u)^{-1}=\sum\limits_{i,j}E_{i,j}\otimes \tilde{t}_{i,j}(u).
\]

Recall that $\Yq$ is a Hopf superalgebra, see \cite{Na99},
with comultiplication $\Delta: \Yq\rightarrow \Yq\otimes \Yq$ given by the formula
\begin{align}\label{comul}
    \Delta(t_{i,j}^{(r)})=\sum\limits_{s=0}^r\sum\limits_{k}(-1)^{(\bar i+\bar k)(\bar j+\bar k)}t_{i,k}^{(s)}\otimes t_{k,j}^{(r-s)}.
\end{align}
For the Lie superalgebra $\fq_1$,
we will choose the basis of $\fq_1$ consisting of two elements,
\begin{align}\label{x and xi}
x:=E_{1,1}+E_{-1,-1},~\xi:=E_{1,-1}+E_{-1,1}.
\end{align}
Due to \eqref{Tijtkl relation} there is a homomorphism
\begin{align}\label{ev map}
\ev: \Yq\rightarrow \U(\fq_1);~t_{1,1}^{(1)}\mapsto -x, t_{1,-1}^{(1)}\mapsto \xi, t_{i,j}^{(r)}\mapsto 0~{\rm for}~r>1.
\end{align}
It is called the {\it evaluation homomorphism} (see \cite[(2.12)]{Na99}, \cite[(4.11)]{PS17}).
Moreover, we write $\Delta_n:\Yq\rightarrow \Yq^{\otimes n}$ for the comultiplication $\Delta$ \eqref{comul} iterated $n-1$ times.
Let
\begin{align}\label{gamman}
 \gamma_n := \ev^{\otimes n}\circ \Delta_n : \Yq \longrightarrow \U(\fq_1)^{\otimes n}
\end{align}
be the homomorphism obtained by applying $\Delta_n$ first then	
applying the evaluation homomorphism \eqref{ev map} to each tensor factor of $\Yq^{\otimes n}$.
In \cite{Na22}, Nazarov proved that
the kernels of all the homomorphisms $\gamma_n$ with $n=1, 2,\dots$ have
zero intersection, i.e.,
\begin{align}\label{intersection=0 for rn}
\bigcap\limits_{n\geq 1}{\rm Ker}\gamma_n=\{0\}.
\end{align}

It is shown in \cite{PS16} that there exists a surjective homomorphism $\varphi_n:\Yq\rightarrow W^n$, see also \cite[Corollary 5.10]{PS17}.
To write the map $\varphi_n$ explicitly, we introduce a family of elements in the universal enveloping algebra $\U(\fh_n)$.
For $1\leq i_1<\cdots i_r\leq n$, set
\[
G_r(i_1,\dots,i_r):=(x_{i_1}+(-1)^{r+1}\xi_{i_1})(x_{i_2}+(-1)^r\xi_{i_2})\cdots (x_{i_r}+\xi_{i_r}).
\]
Composed with the Harish-Chandra map $\vartheta$ \eqref{HC-map},
the surjection $\varphi_n$ is characterized as follows:
\begin{align}\label{hc-varphi-even}
 \vartheta\circ\varphi_n(t_{1,1}^{(r)})=(-1)^{r}\big[\sum\limits_{1\leq i_1<\cdots i_r\leq n}  G_r(i_1,\dots,i_r)\big]_{\rm even},
\end{align}
\begin{align}\label{hc-varphi-odd}
 \vartheta\circ\varphi_n(t_{-1,1}^{(r)})=(-1)^{r}\big[\sum\limits_{1\leq i_1<\cdots i_r\leq n}  G_r(i_1,\dots,i_r)\big]_{\rm odd}.
\end{align}
Both expressions vanish for $r>n$ (see \cite[Theorem 6.2]{PS16} and \cite[Corollary 5.10]{PS17}).
\subsection{The relation between $\varphi_n$ and $\gamma_n$}
For $p=1,\dots,n$,
let $\iota_p:\U(\fq_1)\rightarrow \U(\fq_1)^{\otimes n}$ be the embedding into the tensor product
as the $p$th tensor factor.
Then there is a superalgebra isomorphism
\[
\kappa_n: \U(\fh_n)\xrightarrow{\sim} \U(\fq_1)^{\otimes n};~x_p\mapsto \iota_p(x),~\xi_p\mapsto\iota_p(\xi) .
\]
In view of \cite[page 17]{PS17}, we can also identify $\U(\fh_n)$ with $\U(\fq_1)^{\otimes n}$ by setting
\[
\jmath_n: \U(\fh_n)\xrightarrow{\sim} \U(\fq_1)^{\otimes n};~\jmath_n(x_p)=\iota_{n+1-p}(x),~\jmath_n(\xi_p)=\iota_{n+1-p}(\xi).
\]
An involutive automorphism $\tau_n$ of $\U(\fq_1)^{\otimes n}$ is defined by
\[
 \tau_n: \U(\fq_1)^{\otimes n}\xrightarrow{\sim}\U(\fq_1)^{\otimes n};~c_1\otimes\cdots\otimes c_n\mapsto (-1)^{\sum_{i<j}\deg c_i\deg c_j}c_n\otimes\cdots \otimes c_1,
\]
where all $c_i$ are homogeneous elements.
Then we have
\begin{align}
  \jmath_n=\tau_n\circ \kappa_n.
\end{align}
\begin{Lemma}\label{lamma:j hc varn=gamman}
$\jmath_n\circ \vartheta\circ\varphi_n=\gamma_n$.
\end{Lemma}
\begin{proof}
Remember that $\gamma_n=\ev^{\otimes n}\circ \Delta_n$.
The identification $\jmath_n$ in conjunction with \cite[Theorem 5.14]{PS17} would imply that $\jmath_n^{-1}\circ \gamma_n(t_{i,j}^{(r)})=\vartheta\circ \varphi_n(t_{i,j}^{(r)})$, as desired.
\end{proof}
\section{The center of $\Yq$}\label{sec-center}
\subsection{Poletaeva and Serganova's central elements}
We first recall a family of central elements of $\Yq$ introduced by Poletaeva and Serganova in \cite{PS21}.
For $i\geq 0$, let
\[
\eta_i:=(-\frac{1}{2})^i\ad^i t_{1,1}^{(2)}(t_{1,-1}^{(1)}), Z_{2i}:=\frac{1}{2}[\eta_0,\eta_{2i}],
\]
where $\ad^i t_{1,1}^{(2)}$ is the $i$th power of the adjoint endomorphism $\ad t_{1,1}^{(2)}$.
The elements $\{Z_{2k};~k\geq 0\}$ are algebraically independent generators of the center of $\Yq$ (see \cite[Lemma 5.1]{PS21}).
Write $z_{2k}=\vartheta\circ \varphi_n(Z_{2k})\in \U(\fh_n)$ for $0\leq k\leq n-1$.
Thanks to \cite[Lemma 5.3]{PS21}, the elements $\{z_{2k};~0\leq k\leq n-1\}$ can be expressed in terms of symmetric polynomials of $x_1,\dots,x_n$.
More precisely,
we have
\begin{align}\label{recursion exp of z2i}
z_{2k}=-\sum\limits_{i=1}^{k}\sigma_{2i}z_{2k-2i}+\sigma_{2k+1},
\end{align}
where $\sigma_p=\sigma_p(x_1,\dots,x_n)=\sum_{i_1<\cdots<i_p}x_{i_1}\dots x_{i_p}$ is the elementary symmetric function.
For $k=0$, this means $z_0=\sigma_1=x_1+\cdots+x_n$ (see \cite[Page 145]{PS21}).
Set
\[
E_{\ev}(u)=\sum\limits_{i\geq 0}\sigma_{2i}u^{-2i},\quad E_{\odd}(u):=\sum\limits_{i\geq 0}\sigma_{2i+1}u^{-2i-1}.
\]
By convention, we set $\sigma_0=1$, note that $\sigma_j=0$ for $j>n$.

We define
\begin{align}\label{HuQu}
  H(u):=\sum\limits_{i\geq 0}Z_{2i}u^{-2i-1},~Q(u):=(1+H(u))(1-H(u))^{-1}.
\end{align}
\begin{Proposition}\label{prop:hc-varphin-Hu}
For every $n\geq 1$, we have
\begin{align}\label{hc-phin-EoddoverEev}
\vartheta\circ \varphi_n(H(u))=\frac{E_{\odd}(u)}{E_{\ev}(u)},
\end{align}
\begin{align}\label{hc-phin-Qu}
\vartheta\circ \varphi_n(Q(u))=\prod\limits_{p=1}^n\frac{1+x_pu^{-1}}{1-x_pu^{-1}}.
\end{align}
\end{Proposition}
\begin{proof}
We first prove that \eqref{hc-phin-EoddoverEev}.
Rewrite \eqref{recursion exp of z2i} as
\[
 z_{2k}+\sum\limits_{i=1}^{k}\sigma_{2i}z_{2k-2i}=\sigma_{2k+1}.
\]
Multiplying both sides by $u^{-2k-1}$ and taking sum over $k\geq 0$,
we obtain
\[
\sum\limits_{k\geq 0}z_{2k}u^{-2k-1}+\sum\limits_{k\geq 0}(\sum\limits_{i=1}^{k}\sigma_{2i}z_{2k-2i})u^{-2k-1}=\sum\limits_{k\geq 0}\sigma_{2k+1}u^{-2k-1}.
\]
This yields
\[
(\sum\limits_{i\geq 0}\sigma_{2i}u^{-2i})(\sum\limits_{j\geq 0}z_{2j}u^{-2j-1})= E_{\odd}(u).
\]
Since $\sigma_0=1$, this implies $E_{\ev}(u)$ is invertible,
and \eqref{hc-phin-EoddoverEev} follows.

Now let
\[
E_n(u)=\prod\limits_{p=1}^n(1+x_p u^{-1})=\sum\limits_{j=0}^n\sigma_ju^{-j}.
\]
Then
\[
    E_n(u)=E_{\ev}(u)+E_{\odd}(u),\quad E_n(-u)=E_{\ev}(u)-E_{\odd}(u).
\]
Applying \eqref{hc-phin-EoddoverEev} yields that
\begin{align*}
    \vartheta\circ \varphi_n(Q(u))&=\frac{E_{\ev}(u)+E_{\odd}(u)}{E_{\ev}(u)-E_{\odd}(u)}\\
    &=\frac{E_n(u)}{E_n(-u)}=\prod\limits_{p=1}^n\frac{1+x_pu^{-1}}{1-x_pu^{-1}}.
\end{align*}
\end{proof}
\subsection{The quantum Berezinian}
We put
\[
 C(u):=t_{1,1}(u)\tilde{t}_{1,1}(-u)~{\rm and}~ D(u):=t_{1,-1}(u)\tilde{t}_{1,-1}(u).
\]
Write
\[
 C(u)=1+C_1u^{-1}+C_2u^{-2}+\cdots.
\]
It is shown in \cite{Na22} that all the coefficients $C_i$ are central elements,
and $C_1,C_3,\dots$ freely generate the center of $\Yq$.
The series $C(u)$ will be called the {\it quantum Berezinian} for the Yangian $\Yq$,
see \cite[Section VIII]{Na22}.

Now,
consider the matrix with entries from $\Yq[[u^{-1}]]$,
\begin{align}\label{mathcal-Cu}
\mathcal{C}(u):=\begin{bmatrix}
C(u) & D(u)\\
D(-u) & C(-u)
\end{bmatrix}.
\end{align}
In view of \cite[(69)]{Na22},
the homomorphism $\gamma_n$ \eqref{gamman} maps the matrix $\CC(u)$ to the product over $p=1,\dots,n$ of the matrices
\begin{align}\label{matrix gamma image}
\frac{1}{u^2-\iota_p(x)^2-\iota_p(x)}\begin{bmatrix}
(u-\iota_p(x))^2& -\iota_p(x)\\
-\iota_p(x) & (u+\iota_p(x))^2
\end{bmatrix}.
\end{align}
Since $\deg \iota_p(x)=0$ relative to the $\ZZ_2$-grading on $\U(\fq_1)^{\otimes n}$,
the matrices \eqref{matrix gamma image} commute,
so the ordering of the factors in the product does not matter.
Consequently, Lemma \ref{lamma:j hc varn=gamman} implies that
\begin{align}\label{image-hc-varphin-Cu}
\vartheta\circ\varphi_n(\CC(u))=\prod\limits_{p=1}^n M_{x_p}(u),
\end{align}
where
\[
 M_{x_p}(u)=\frac{1}{u^2-x_p^2-x_p}\begin{bmatrix}
(u-x_p)^2& -x_p\\
-x_p & (u+x_p)^2
\end{bmatrix}, \quad p=1,\dots,n.
\]

Now we introduce more formal series, we let
\begin{equation}\label{delta u}
\begin{aligned}
\delta(u)=\sqrt{1+4u^2}=2u\sqrt{1+\frac{1}{4u^2}}=&2u+\frac{1}{4}u^{-1}-\frac{1}{64}u^{-3}+\cdots\\
=&2u(1+\frac{1}{8}u^{-2}-\frac{1}{128}u^{-4}+\cdots),
\end{aligned}
\end{equation}
and set
\begin{align}\label{rho u}
\rho(u):=\frac{\delta(u)-1}{2}.
\end{align}
We denote by $\Id$ the $2\times 2$ identity matrix and define
\begin{align*}
N(u)=\begin{bmatrix}
2u & 1\\
1& -2u
\end{bmatrix}.
\end{align*}
Note that $N(u)^2=(1+4u^2)\Id$.
We further set
\[
P_+(u)=\frac{1}{2}\left(\Id+\frac{1}{\delta(u)}N(u)\right),\quad P_-(u)=\frac{1}{2}\left(\Id-\frac{1}{\delta(u)}N(u)\right).
\]
Here,
$\displaystyle\frac{1}{\delta(u)}$ is understood as
\begin{align*}
\frac{1}{\delta(u)}=&\frac{1}{2u}(1+\frac{1}{8}u^{-2}-\frac{1}{128}u^{-4}+\cdots)^{-1}\\
=&\frac{1}{2}u^{-1}+\frac{1}{16}u^{-3}+\cdots.
\end{align*}
Then
\begin{equation}\label{p-zhengfu-relation}
\begin{aligned}
P_+(u)^2=&P_+(u),\quad P_-(u)^2=P_-(u),\\
P_+(u)P_-(u)=&0=P_-(u)P_+(u),\quad P_+(u)+P_-(u)=\mathrm{Id}.
\end{aligned}
\end{equation}

\begin{Lemma}\label{lemma:product Mxu decomposition}
\[
\prod\limits_{p=1}^n M_{x_p}(u)=\left(\prod\limits_{p=1}^n\frac{\rho(u)+1-x_p}{\rho(u)+1+x_p}\right)P_+(u)+\left(\prod\limits_{p=1}^n\frac{\rho(u)+x_p}{\rho(u)-x_p}\right)P_-(u).
\]
\end{Lemma}
\begin{proof}
Let $p\in\{1,\dots,n\}$.
First we observe that
\[
 M_{x_p}(u)=\frac{(u^2+x_p^2)\Id-x_pN(u)}{u^2-x_p^2-x_p}.
\]
It is easy to verify that $N(u)P_+(u)=\delta(u)P_+(u)$ and $N(u)P_-(u)=-\delta(u)P_-(u)$.
This readily implies that
\begin{align*}
M_{x_p}(u)P_+(u)=&\frac{u^2+x_p^2-x_p\delta(u)}{u^2-x_p^2-x_p}P_+(u)\\
=&\frac{(x_p-\rho(u))(x_p-\rho(u)-1)}{(\rho(u)-x_p)(\rho(u)+x_p+1)}P_+(u)=\frac{\rho(u)+1-x_p}{\rho(u)+1+x_p}P_+(u)
\end{align*}
and
\begin{align*}
M_{x_p}(u)P_-(u)=&\frac{u^2+x_p^2+x_p\delta(u)}{u^2-x_p^2-x_p}P_-(u)\\
=&\frac{(x_p+\rho(u))(x_p+\rho(u)+1)}{(\rho(u)-x_p)(\rho(u)+x_p+1)}P_-(u)=\frac{\rho(u)+x_p}{\rho(u)-x_p}P_-(u).
\end{align*}
Since $P_+(u)+P_-(u)=\Id$, it follows that
\[
 M_{x_p}(u)=\frac{\rho(u)+1-x_p}{\rho(u)+1+x_p}P_+(u)+\frac{\rho(u)+x_p}{\rho(u)-x_p}P_-(u).
\]
This in combination with \eqref{p-zhengfu-relation} proves the Lemma.
\end{proof}

\begin{Proposition}\label{prop:main prop-1}
The following relation holds in $\Mat_2(\Yq)[[u^{-1}]]$:
\begin{align}\label{main prop-matrix=}
\CC(u)=Q(\rho(u))P_-(u)+[Q(\rho(u)+1)]^{-1}P_+(u).
\end{align}
\end{Proposition}
\begin{proof}
According to \eqref{image-hc-varphin-Cu} and Lemma \ref{lemma:product Mxu decomposition} we have
\[
\vartheta\circ\varphi_n(\CC(u))=\left(\prod\limits_{p=1}^n\frac{\rho(u)+x_p}{\rho(u)-x_p}\right)P_-(u)+\left(\prod\limits_{p=1}^n\frac{\rho(u)+1-x_p}{\rho(u)+1+x_p}\right)P_+(u).
\]
Moreover, Proposition \ref{prop:hc-varphin-Hu} implies that the right hand side is equal to
\[
\vartheta\circ\varphi_n(Q(\rho(u))P_-(u)+[Q(\rho(u)+1)]^{-1}P_+(u)).
\]
Since this holds for any $n\geq 1$, our assertion immediately from \eqref{intersection=0 for rn}.
\end{proof}

\begin{Theorem}
The following relations hold in $\Yq[[u^{-1}]]$:
\begin{align}\label{Q-rhou=CD}
Q(\rho(u))=C(u)-(\delta(u)+2u)D(u),
\end{align}
\begin{align}\label{Qrho-1=CD}
[Q(\rho(u)+1)]^{-1}=C(u)+(\delta(u)-2u)D(u).
\end{align}
\end{Theorem}
\begin{proof}
We put
\[
a:=\frac{\delta(u)-2u}{2\delta(u)},\quad b=\frac{\delta(u)+2u}{2\delta(u)}.
\]
Comparing the $(1,1)$ and $(1,2)$ entries in \eqref{main prop-matrix=} yields
\[
C(u)=aQ(\rho(u))+b[Q(\rho(u)+1)]^{-1}
\]
and $2\delta(u)D(u)=Q(\rho(u)+1)^{-1}-Q(\rho(u))$.
It follows that
\[
C(u)=(a+b)Q(\rho(u))+2b\delta(u)D(u).
\]
Note that $a+b=1$ and $2b\delta(u)=\delta(u)+2u$.
We thus obtain \eqref{Q-rhou=CD}, and \eqref{Qrho-1=CD} follows similarly.
\end{proof}

By solving the equations \eqref{Q-rhou=CD} and \eqref{Qrho-1=CD}, we obtain the following corollary:
\begin{Corollary}
The following relations hold in $\Yq[[u^{-1}]]$:
\begin{align}
C(u)=&\frac{\delta(u)-2u}{2\delta(u)}Q(\rho(u))+\frac{\delta(u)+2u}{2\delta(u)}[Q(\rho(u)+1)]^{-1},\label{Cu=Qu}\\
D(u)=&\frac{[Q(\rho(u)+1)]^{-1}-Q(\rho(u))}{2\delta(u)}.\label{Du=Qu}
\end{align}
\end{Corollary}
\begin{Remark}\label{Remark}
\upshape
(1) Recall that \eqref{HuQu}, \eqref{delta u} and \eqref{rho u}.
It is easy to see that
\[
Q(-u)=Q(u)^{-1},\quad \delta(-u)=-\delta(u),\quad \rho(-u)=-(\rho(u)+1).
\]
Hence,
by replacing $u$ by $-u$ in \eqref{Cu=Qu} and \eqref{Du=Qu},
we obtain $D(u)=D(-u)$ and
\[
C(u)-C(-u)=4uD(u).
\]
These relations were first obtained by Nazarov, cf. \cite[(70)]{Na22}.

(2) In \cite{Na99}, Nazarov investigated the center of the super Yangian $\Y(\fq_n)$.
For the Yangian $\Yq$, we let
\[
    \cZ(u):=t_{1,1}(u)\tilde{t}_{1,1}(u)+t_{1,-1}(-u)\tilde{t}_{1,-1}(u).
\]
Then relations \eqref{tij-u=t-i-iu} imply that $\cZ(u)=\cZ(-u)$.
Thus
\[
    \cZ(u)=1+\cZ_2u^{-2}+\cZ_4u^{-4}+\cdots.
\]
The elements $\cZ_2,\cZ_4,\dots$ are free generators of the center of $\Yq$ (\cite[Theorem 3.4]{Na99}).
Moreover, in terms of the quantum Berezinian,
we have $\cZ(u)=C(u)C(-u)-D(u)D(-u)$ (\cite[(68)]{Na22}).
Note that $\cZ(u)$ is equal to the determinant of the matrix $\CC(u)$.
The two eigenvalues of $\CC(u)$ in \eqref{main prop-matrix=} are $Q(\rho(u))$ and $[Q(\rho(u)+1)]^{-1}$.
Therefore we obtain
\[
\cZ(u)=Q(\rho(u))[Q(\rho(u)+1)]^{-1}.
\]
\end{Remark}
\bigskip
\noindent
{\large\textbf{Declarations}}\\
\\
\textbf{Competing interests}\\
The authors have no competing interests to declare that are relevant to the content of this article.
\\
\\
\textbf{Acknowledgment}\\
We would like to thank the anonymous reviewer for his/her helpful comments.
This work is supported by the
National Natural Science Foundation of China (12671035),
the Natural Science Foundation of
Hubei Province (No. 2025AFB716) and the Fundamental Research Funds for the Central
Universities (No. XJ2026002101).
In the preparation of this paper, we collaborated with ChatGPT 5.6 Sol. ChatGPT assisted us in clarifying the relation between the homomorphisms $\varphi_n$ and $\gamma_n$,
in particular the formulation and understanding of Lemma \ref{lamma:j hc varn=gamman}.
All statements
and proofs were independently checked and are the sole responsibility of the authors.
\\
\\
\textbf{ Availability of data and materials}
\\
No data was used for the research described in the article

\end{document}